\documentclass[11pt]{amsart}
\usepackage[utf8]{inputenc}
\usepackage{amsmath,amssymb}
\usepackage{wrapfig}
\usepackage{url}
\usepackage{mathtools}
\usepackage{graphicx}
\usepackage{stmaryrd}
\usepackage{amsthm}
\usepackage{xcolor}
\usepackage[colorlinks=true,linkcolor=blue,citecolor=blue]{hyperref}
\usepackage[shortlabels]{enumitem}
\usepackage{comment}
\usepackage{relsize}
\usepackage{dsfont}
\usepackage{stmaryrd}
\usepackage{geometry}
\usepackage{setspace}
\usepackage{mathrsfs} 
\usepackage{orcidlink}
\usepackage{slashed}
\usepackage{yhmath}
\makeatletter
\newtheorem{theorem}{Theorem}
\numberwithin{theorem}{section} 
\numberwithin{equation}{section}
\newtheorem{corollary}[theorem]{Corollary}

\newtheorem{prop}[theorem]{Proposition} 
\theoremstyle{definition}
\newtheorem{example}{Example}
\numberwithin{example}{section}

\newtheorem{lemma}[theorem]{Lemma}
\theoremstyle{definition}
\newtheorem{obs}{Remark}
\theoremstyle{definition}

\newcommand{\R}{\mathbb{R}}
\newcommand{\Z}{\mathbb{Z}}
\newcommand{\N}{\mathbb{N}}
\newcommand{\T}{\mathbb{T}}
\newcommand{\C}{\mathbb{C}}
\newcommand{\Tr}{\operatorname{Tr}}
\newcommand{\HS}{\operatorname{HS}}
\title{Global hypoellipticity of systems of Fourier multipliers on compact Lie groups}

\author[A. Kowacs]{André Pedroso Kowacs} 
\address{
  Department of Mathematics
  \endgraf
 Universidade Federal do Paraná (UFPR), Brazil
  }
\email{andrekowacs@gmail.com}



\date{\today}

\begin{document}

\begin{abstract}
We apply the characterization of global hypoellipticity for $G$-invariant operators on homogeneous vector bundles obtained by Cardona and Kowacs [J. Pseudo-Differ. Oper. Appl. 16, 23 (2025)] to obtain a necessary and sufficient condition for an arbitrary system of left-invariant operators on a compact Lie group to be globally hypoelliptic, providing a full proof independent of the bundle structure of that paper. We then prove alternative sufficient conditions for globally hypoellipticity for a large class of particular cases of systems making use of lower bounds for the smallest singular value of complex matrices.
\end{abstract}
\maketitle
\allowdisplaybreaks

\section{Introduction}

Global properties of (pseudo)differential operators on smooth compact manifolds have been intensively studied in the last few years. One of these properties is known as {\it global hypoellipticity}, which is characterized by the fact that an operator $P$ which acts on distributions is globally hypoelliptic if it only maps distributions given by smooth functions into smooth functions, that is, if $Pu$ is smooth, then $u$ is smooth. This property has been extensively studied in many different settings, such as in the torus (see for instance  \cite{BergamascoKirilovGH,PsudoGHKirilov,PseudoTorus,GF1,GF2,GFremarks}), other compact Lie groups (see for instance  \cite{Araujo,KirilovTS,KirilovCOMP,RuzTurWirth}) and also in general compact manifolds (see for instance \cite{bergack,hounie,houniezugli,Kirilov_2020}). It is worth mentioning the existence of variants of this property which replace the smoothness in the definition by other types of regularity, such analytic or Gevrey.

The study of global hypoellipticity can also be extended to system of operators. In this case the definition is analogous: a system of operators is globally hypoelliptic if whenever a vector of smooth functions admits a vector of solutions for such a system, then this vector of solutions must be smooth. The problem of finding necessary and sufficient conditions for global hypoellipticity of systems has attracted much attention in the last few decades, in part due to its connection with global hypoellipticity involving differential forms,
and general compact manifolds (as in \cite{BergamascoGH,triangulariz,overdetermined}). More recently, by making use of the Fourier analysis developed in \cite{Ruz}, there has been some development in obtaining precise results for some particular systems on compact Lie groups also (see \cite{diag-systems}).

In this paper we obtain necessary and sufficient conditions for the globally hypoellipticity of an arbitrary system of Fourier multiplier (left-invariant continuous linear operators) acting on smooth functions of a compact Lie group. This is achieved by considering the system of operators as a single operator acting between smooth vector-valued functions and using the Fourier analysis for vector-valued functions on compact Lie groups developed in \cite{HomoVector}.

This paper is structured as follows. In Section \ref{sec2} we recall the main theory and introduce the basic notation used throughout the paper. In Section \ref{sec3} we present our main results and their proofs, along with some illustrative examples.

\section{Preliminaries}\label{sec2}

\subsection{Fourier analysis on compact Lie groups}
For a more detailed exposition of the Fourier analysis on compact Lie groups, see \cite{HomoVector} and \cite{Ruz}. Let $G$ be a compact Lie group and denote by $\widehat{G}$ its unitary dual. This set consists of equivalence classes of all continuous unitary representations of $G$, where two representations $\xi,\eta$ are equivalent if there exists a linear bijection $A$ such that $\eta(g)A=A\xi(g)$. 

Since $G$ is compact, the representations in each class $[\xi]\in\widehat{G}$ are finite dimensional, and we denote their common dimension by $d_\xi$. From now on we always fix one a representative matrix-valued representative of each $[\xi]\in\widehat{G}$ . It follows from the Peter-Weyl Theorem that the collection of all coefficient functions of elements of $\widehat{G}$ is an orthogonal Schauder basis for $L^2(G)$, where integration is taken with respect to the Haar measure in $G$. The matrix-valued Fourier coefficients of $f\in L^2(G)$ as defined in \cite{Ruz} are given by 
\begin{equation*}
    \widehat{f}(\xi)\doteq \int_G f(x)\xi(x)^*dx\in\mathbb{C}^{d_\xi\times d_\xi}, 
\end{equation*}
for every $ [\xi]\in\widehat{G}$.
This gives the Fourier inversion formula as
\begin{equation*}
    f(x)=\sum_{[\xi]\in\widehat{G}} d_\xi\Tr\left(\xi(x)\widehat{f}(\xi)\right),
\end{equation*}
for almost every $x$ in $G$. There is also an analogue of Plancherel's Theorem, namely
\begin{equation*}
    \|f\|_{L^2(G)}=\left(\int_G|f(x)|^2dx\right)^{\frac{1}{2}}=\left(\sum_{[\xi]\in \widehat{G}}d_\xi\|\widehat{f}(\xi)\|_{\HS}^2\right)^{\frac{1}{2}},
\end{equation*}
for any $f\in L^2(G)$, where $\|A\|_{\HS}$ denotes the Hilbert-Schmidt norm of a matrix $A$, given by $\sqrt{\Tr(A^*A)}=\sqrt{\sum_{i,j}|A_{ij}|^2}$.

The Fourier transform extends to the set $\mathcal{D}'(G)$ of distributions on $G$, via the formula
\begin{equation*}
    \widehat{u}(\xi)\doteq\langle u,\xi^*\rangle\in\mathbb{C}^{d_\xi\times d_\xi}, 
\end{equation*}
for every $u\in\mathcal{D}'(G) \text{ and } [\xi]\in\widehat{G}$, where the evaluation should be understood coefficient-wise.
Given a continuous linear operator $D:C^\infty(G)\to C^\infty(G)$, define its matrix-valued global symbol $\sigma_D$ by
\begin{equation*}
    \sigma_D(x,\xi)\doteq \xi(x)^*D\xi(x),
\end{equation*}
for all $(x,[\xi])\in G\times\widehat{G}$. In \cite{Ruz} the authors proved the quantisation formula
\begin{equation}\label{quantisation} Df(x)=\sum_{[\xi]\in\widehat{G}}d_\xi\Tr\left[\xi(x)\sigma_D(x,\xi)\widehat{f}(\xi)\right]
\end{equation}
for all $x\in G$, $f\in C^\infty(G)$. One can prove that the symbol $\sigma_D$ in the formula above does not depend on $x\in G$ if and only if $D$ is a left-invariant operator, that is, it commutes with  left translations under the group action on itself. In this case we write $\sigma_D(x,\xi)=\sigma_D(\xi)$ and consequently $\widehat{Df}(\xi)=\sigma_D(\xi)\widehat{f}(\xi)$, for every $[\xi]\in\widehat{G}$ and we call it a Fourier multiplier.

Next, let $n\in\N$ and consider the canonical basis $\{e_1,\dots e_n\}$ in $\mathbb{C}^n$, viewed as a $\C$-vector space. Given a mapping $f:G\to \C^n$, we denote its components by $f_i:G\to \mathbb{C}$, 
 where $f_i(x)=\langle f(x),e_{i}\rangle_{\C^n}\in\mathbb{C}$, for each $x\in G$, $1\leq i\leq n$. Following \cite{HomoVector}, for $f\in L^2(G,\C^n)$ we define the Fourier coefficients of $f$ by 
 \begin{equation*}
     \widehat{f}(\xi)\doteq\begin{pmatrix}
         \widehat{f_1}(\xi)\\\vdots\\ \widehat{f_n}(\xi)
     \end{pmatrix},
 \end{equation*}
 for every $[\xi]\in\widehat{G}$, that is, $\widehat{f}(\xi)$ is a $n\times 1$ block column matrix with $d_\xi\times d_\xi$ blocks. This definition also extends to the set of distributions $\mathcal{D}'(G,\C^n)$ by  
 \begin{equation*}
  \widehat{u}(\xi)\doteq \begin{pmatrix}
      \widehat{u_1}(\xi)\\
      \vdots \\
      \widehat{u_n}(\xi)
  \end{pmatrix}\in (\mathbb{C}^{d_\xi\times d_\xi})^{n\times 1},
 \end{equation*}
 where 
 \begin{align*}
     \widehat{u_i}(\xi)_{\alpha\beta}\doteq\langle u,\overline{\xi_{\beta\alpha}}\otimes e_i\rangle
 \end{align*}
  and $\overline{\xi_{\beta\alpha}}\otimes e_i\in C^\infty(G,\C^n)$ is given by $\overline{\xi_{\beta\alpha}}\otimes e_i(x) = \overline{\xi_{\beta\alpha}}(x)e_i$, for $1\leq \alpha,\beta\leq d_\xi$ and  $1\leq i\leq n$. 
  
  We can characterize smooth functions and distributions by the decay of their Fourier coefficients, as in the scalar-valued case, as follows.
\begin{prop}\label{prop_fourier_smooth}
    Given a sequence of block column matrices $\widehat{v}(\xi)\in(\mathbb{C}^{d_\xi\times d_\xi})^{n\times 1}$, $[\xi]\in\widehat{G}$, they correspond to the Fourier coefficients of a function in $C^\infty(G,\C^n)$ if and only if for every $N>0$, there exists $C_N>0$ such that 
    \begin{equation*}
        \|\widehat{v}(\xi)\|_{\HS}\leq C_N\langle \xi\rangle^{-N},
    \end{equation*}
     for every $[\xi]\in\widehat{G}$.
\end{prop}
\begin{proof}
    Indeed, this follows directly from the following facts: 
    \begin{enumerate}
        \item A function $f:G\to C^n$ is smooth if and only if $f_i:G\to \C^m$ is smooth, for every $1\leq i\leq n$,
        \item A sequence of matrices $\widehat{v}(\xi)\in \mathbb{C}^{d_\xi\times d_\xi}$, $[\xi]\in\widehat{G}$, corresponds to the Fourier coefficients of a function in $C^\infty(G)$ if and only if for every $N>0$, there exists $C_N>0$ such that 
    \begin{equation*}
        \|\widehat{v}(\xi)\|_{\HS}\leq C_N\langle \xi\rangle^{-N}.
    \end{equation*}
     for every $[\xi]\in\widehat{G}$ (see \cite{Ruz}).
        \item For $f\in L^1(G,\C^n)$ and $N>0$, $\|\widehat{f}(\xi)\|_{\HS}\leq C_N\langle \xi\rangle^{-N}$ for some $C_N>0$ and every $[\xi]\in\widehat{G}$ if and only if $\|\widehat{f_i}(\xi)\|_{\HS}\leq C_{i,N}'\langle \xi\rangle^{-N}$ for every $[\xi]\in\widehat{G}$ for some $C_{i,N}'>0$ and every $1\leq i\leq n$,  $[\xi]\in\widehat{G}$.
        \end{enumerate}
\end{proof}
\begin{prop}\label{prop_fourier_distrib}
    Given a sequence of block column matrices $\widehat{v}(\xi)\in(\mathbb{C}^{d_\xi\times d_\xi})^{n\times 1}$, $[\xi]\in\widehat{G}$, they correspond to the Fourier coefficients of a distribution in $\mathcal{D}'(G,\C^n)$ if and only if there exist $C,N>0$ such that 
    \begin{equation*}
        \|\widehat{v}(\xi)\|_{\HS}\leq C\langle \xi\rangle^{N},
    \end{equation*}
     for every $[\xi]\in\widehat{G}$ and $1\leq i\leq n$.
     \end{prop}
\begin{proof}
    Notice that every linear functional $u$ over smooth vector valued functions can be written as
    \begin{equation*}
        \langle u , f\rangle = \langle \tilde u_1,f_1\rangle + \dots \langle \tilde u_n,f_n\rangle,
    \end{equation*}
    where each linear functional $\tilde u_i:C^\infty(G)\to \mathbb{C}$ is given by $\langle \tilde u_i,f_i\rangle =\langle u, f_i\otimes e_i\rangle$, $1\leq i\leq n$. Moreover, $u$ is continuous if and only if each $\tilde u_i$ is continuous and when these are all continuous, our definition of $\widehat{u_i}(\xi)$ is consistent with the Fourier transform of $\tilde u_i$, for every $[\xi]\in\widehat{G}$. Hence, the proof is analogous to the proof of Proposition \ref{prop_fourier_smooth}, in the sense that it is a consequence of the scalar-valued case proved in \cite{Ruz}.
\end{proof}

For every $s\in\R$, the Sobolev space  $H^{s}(G,\C^n)$ is defined to be the set of all $u\in \mathcal{D}'(G,\C^n)$ such that Sobolev norm
\begin{align}\label{eqsobonorm}
    \|u\|_{H^{s}(G,\C^n)}
    &\doteq \left(\sum_{[\xi]\in\widehat{G}}d_\xi\langle \xi\rangle^{2s} \sum_{i=1}^{n}\|\widehat{u_i}(\xi)\|_{\HS}^2\right)^{\frac{1}{2}}\\
&=\left(\sum_{[\xi]\in\widehat{G}}d_\xi\langle \xi\rangle^{2s} \|\widehat{u}(\xi)\|_{\HS}^2\right)^{\frac{1}{2}}\notag
\end{align}
is finite. Here, $\langle\xi\rangle = (1+\nu_\xi)^{\frac{1}{2}}$ denotes the common eigenvalue of the operator $(\text{Id}+\mathcal{L}_G)^{\frac{1}{2}}$ corresponding to the coefficient functions of $[\xi]\in\widehat{G}$, and where $\mathcal{L}_G$ is the positive Laplace-Beltrami operator on $G$. 

Taking into account the growth of the eigenvalues $\langle\xi\rangle$ (see \cite{Ruz}), we obtain as a consequence of propositions \ref{prop_fourier_smooth} and \ref{prop_fourier_distrib} that
\begin{equation}\label{sobolev-cup-cap}
  \bigcap_{s\in\R} H^s(G,\C^n)=C^\infty(G,\C^n) \text{ and } \bigcup_{s\in\R} H^s(G,\C^n)=\mathcal{D}'(G,\C^n).  
\end{equation}

  Let $m\in \N$ and consider a continuous linear operator $D:C^\infty(G,\C^n)\to C^\infty(G,\C^m)$. Define its matrix-valued symbol by
\begin{equation}\label{symbol}
    \sigma_D(i,j,x,\xi) = \xi(x)^*e_j^*[D(\xi\otimes e_i)(x)],
\end{equation}
for $1\leq i\leq n$, $1\leq j\leq m$, $x\in G$ and $[\xi]\in\widehat{G}$, where $e_j^*[v]\doteq\langle v,e_j\rangle_{\C^m}$, for $v\in\C^m$. Note that here we use the same notation for the canonical basis vectors of $\C^m$, that is,  $\{e_1,\dots,e_m\}$, as no confusion should arise. In \cite{HomoVector} the authors proved that for every $f\in C^\infty(G,\C^n)$, we have that
\begin{equation}\label{quantisation1}
    Df(x) = \sum_{j=1}^m\left(\sum_{[\xi]\in\widehat{G}}d_\xi\Tr\left[\sum_{i=1}^n\sigma_D(i,j,x,\xi)\widehat{f}_i(\xi)\xi(x)\right]\right)e_j,
\end{equation}
 for every $x\in G$. 
 This formula can also be written as 
\begin{equation*}
    Df(x) = \sum_{[\xi]\in\widehat{G}}d_\xi\Tr[\sigma_D(x,\xi)\widehat{f}(\xi)\xi(x)],
\end{equation*}
where $\sigma_D(x,\xi)\in (\mathbb{C}^{d_\xi\times d_\xi})^{m\times n}$ is the block matrix with blocks given by: 
\begin{equation*}
    \sigma_D(x,\xi)_{ij}=\sigma_D(j,i,x,\xi)
\end{equation*}
    for $ 1\leq j\leq n,\ 1\leq i\leq m$ and $ (x,[\xi])\in G\times \widehat{G}$, and the trace for the $m\times 1$ block matrices in the formula  above should be understood component-wise. A similar formula also holds for $f\in\mathcal{D}'(G,\C^n)$, with convergence in the sense of distributions.

Once again, the symbol $\sigma_D$ above does not depend on $x$ if and only if $D$ is left-invariant, and similarly to the scalar case we write $\sigma_D(x,\xi)=\sigma_D(\xi),\ \sigma_D(i,j,x,\xi)=\sigma_D(i,j,\xi)$. As a consequence of the quantization formula \eqref{quantisation1}, we obtain that in this case $\widehat{(Df)_j}(\xi)=\sum_{i=1}^n\sigma_D(i,j,\xi)\widehat{f_i}(\xi)$, or equivalently, 
$\widehat{Df}(\xi)=\sigma_D(\xi)\widehat{f}(\xi)$, for every $[\xi]\in\widehat{G}$, and we also say that $D$ is a Fourier multiplier.


\section{Main Results and Applications}\label{sec3}

In this section we apply the theory developed in Section \ref{sec2} to obtain a characterization of the global hypoellipticity of systems of left-invariant pseudo-differential operators on compact Lie groups. We note that this result can also be seen as a consequence of \cite[Theorem 3.2]{CardKow}, however we choose to present a proof that does not rely on the homogeneous vector bundle structure.  We then explore a few particular cases of systems which allow for obtaining other sufficient conditions for global hypoellipticity.

\noindent First, we recall the definition of global hypoellipticity on compact Lie groups.

Let $G$ be a compact Lie group, $D:C^\infty(G)\to C^\infty(G)$ a continuous linear operator. The operator $D$ is said to be globally hypoelliptic if
$$u\in\mathcal{D}'(G)\text{ and } Du=f\in C^\infty(G)\implies u\in C^\infty(G).$$
\noindent This definition extends naturally to vector-valued functions, as follows:

Let $n,m\in\N$, $G$ be a compact Lie group and $D:C^\infty(G,\C^n)\to C^\infty(G,\C^m)$ a continuous linear operator. The operator $D$ is said to be globally hypoelliptic if
$$u\in\mathcal{D}'(G,\C^n)\text{ and } Du=f\in C^\infty(G,\C^m)\implies u\in C^\infty(G,\C^n).$$

\begin{theorem}\label{theo-vector}
    Let $n,m\in\N$, $G$ a compact Lie group and $D:C^\infty(G,\mathbb{C}^n)\to C^\infty(G,\mathbb{C}^m)$ a left-invariant continuous linear operator. Then $D$ is globally hypoelliptic if and only if there exist $k\in\R$ and $C>0$ such that
    \begin{equation}\label{ineq_theo}
        \lambda_{\min}[\sigma_D(\xi)]\geq C\langle\xi\rangle^k,
    \end{equation}
    for all but finitely many $[\xi]\in\widehat{G}$, where $\lambda_{\min}[\sigma_D(\xi)]$ denotes the smallest singular value of the block matrix $\sigma_D(\xi)$, for $[\xi]\in\widehat{G}$.
\end{theorem}

\begin{proof}
    First assume that inequality \eqref{ineq_theo} holds for all $[\xi]\in\widehat{G}\backslash V$, where $V\subset\widehat{G}$ is finite. Then for $u\in \mathcal{D}'(G,\C^n)$ such that $Du\in C^\infty(G,\C^m)$, and for any $s\in\mathbb{R}$, we have that
    \begin{align}\label{ineqprooftheovector1}
        +\infty>\|Du\|^2_{H^s(G,\C^m)} &= \sum_{[\xi]\in\widehat{G}}d_\xi\langle\xi\rangle^{2s} \sum_{j=1}^{m}\|\widehat{(Du)_j}(\xi)\|_{\HS}^2\notag\\
        &=\sum_{[\xi]\in\widehat{G}}d_\xi\langle\xi\rangle^{2s}\sum_{j=1}^{m}\left\|\sum_{i=1}^{n }\sigma_D(i,j,\xi)\widehat{u_i}(\xi)\right\|_{\HS}^2\notag \\
        &=\sum_{[\xi]\in\widehat{G}}d_\xi\langle\xi\rangle^{2s}\left\|\sigma_D(\xi)\widehat{u}(\xi)\right\|_{\HS}^2.
    \end{align}
        But notice that, 
        \begin{align*}
          \left\|\sigma_D(\xi)\widehat{u}(\xi)\right\|_{\HS}
           &\geq 
           \lambda_{\min}[\sigma_D(\xi)]\|\widehat{u}(\xi)\|_{\HS},
        \end{align*}
        for every $[\xi]\in\widehat{G}$, by Lemma \ref{lemmasingvalue} (proven below).
        Hence \eqref{ineq_theo} and \eqref{ineqprooftheovector1} imply that
    \begin{align*}    
         +\infty>\|Du\|^2_{H^s(G,\C^m)}&\geq\sum_{[\xi]\in \widehat{G}\backslash V} d_\xi\langle\xi\rangle^{2s} \lambda_{\min}[\sigma_D(\xi)]^2\|\widehat{u}(\xi)\|_{\HS}^2\\
        & \geq C^2\sum_{[\xi]\in \widehat{G}\backslash V}d_\xi\langle\xi\rangle^{2(s+k)}\|\widehat{u}(\xi)\|_{\HS}^2.
    \end{align*}
    Since $V$ is finite, this implies that $\|u\|_{H^{s+k}(G,\C^n)}<+\infty$. Because this holds for any $s\in\mathbb{R}$, we conclude by \eqref{sobolev-cup-cap} that $u\in C^\infty(G,\C^n)$ and thus $D$ is globally hypoelliptic. 
     Suppose now that inequality \eqref{ineq_theo} does not hold. Note that by the characterization of the smallest singular value, we can write 
     \begin{align*}
         \lambda_{\min}[\sigma_D(\xi)]&=\min\left\{ \|\sigma_D(\xi)v(\xi)\|_2:v(\xi)\in (\C^{d_\xi\times 1})^{n\times 1},\,\|v(\xi)\|_2=1 \right\}\\
         &=\min\left\{\left(\sum_{j=1}^m\left\|\sum_{i=1}^n\sigma_D(i,j,\xi)v(i,\xi)\right\|_2^2\right)^{\frac{1}{2}}:v(i,\xi)\in \C^{d_\xi\times 1},\,\sum_{i=1}^n\|v(i,\xi)\|_2^2=1 \right\}.
     \end{align*}
     So, for each $\ell\in\mathbb{N}$, there exist distinct $[\xi_\ell]\in\widehat{G}$ and $v(i,\xi_\ell) \in \C^{d_\xi\times 1},\,1\leq i\leq n$ satisfying $\sum_{i=1}^n\|v(i,\xi)\|_2^2=1$  such that
    \begin{equation*}
\sum_{j=1}^m\left\|\sum_{i=1}^n\sigma_D(i,j,\xi)v(i,\xi)\right\|_2^2< \langle \xi_\ell\rangle^{-\ell},
    \end{equation*}
    for every $\ell\in \N$. Let $u\in \mathcal{D}'(G,\C^n)$ be defined by the Fourier coefficients
    \begin{align*}
        \widehat{u_i}(\xi) =\begin{cases}
            \begin{pmatrix}
                \lvert&0&\cdots\\
                v(i,\xi_\ell)&0&\cdots\\
                \lvert&0&\cdots
            \end{pmatrix}_{d_\xi\times d_\xi}&\text{if } \xi=\xi_\ell,\ \ell\in\N, \text{ and }\,1\leq i\leq n,\\
            0_{d_\xi\times d_\xi}&\text{otherwise.}
        \end{cases} 
    \end{align*}
     Then $u\in\mathcal{D}'(G,\C^n)\backslash C^\infty(G,\C^n)$, as $\|\widehat{u}(\xi_\ell)\|_{\HS}^2=\sum_{i=1}^n\|v(i,\xi)\|_2^2=1$, for all $\ell\in\mathbb{N}$, and $\widehat{u}(\xi)=0$ for every other $[\xi]\in \widehat{G}$, but on the other hand
      \begin{align*}
        \|\widehat{{Du}}(\xi_\ell)\|_{\HS}^2&=\sum_{j=1}^{m}\left\|\sum_{i=1}^{n}\sigma_{{D}}(i,j,\xi_\ell)\widehat{u_i}(\xi_\ell)\right\|_{\HS}^2 \\
        &=\sum_{j=1}^{m}\left\|\sum_{i=1}^{n}\sigma_{{D}}(i,j,\xi_\ell)v(i,\xi_\ell)\right\|_{2}^2\\
        &<\langle \xi_\ell\rangle^{-\ell},
    \end{align*}
    for all $\ell\in\N$, while $\widehat{{Du}}(\xi)=0$ for all other $[\xi]\in\widehat{G}$,
    therefore $Du\in C^\infty(G,\C^m)$ and so $D$ is not globally hypoelliptic.
\end{proof}
\begin{lemma}\label{lemmasingvalue}
Let $A\in\mathbb{C}^{r\times s}$ and $B\in\mathbb{C}^{s\times p}$, where $r,s,p\in\N$, be two complex matrices. Then 
\begin{equation*}\label{ineqsingularvalue}
        \|AB\|_{HS}\geq \lambda_{\min}[A]\|B\|_{HS}.
    \end{equation*}
\end{lemma}
\begin{proof}
Since $A^*A$ is a normal square matrix, it follows from the spectral theorem that we can write  $A^*A=Q^*\Lambda Q$, where $\Lambda\in\mathbb{C}^{s\times s}$ is a diagonal matrix whose entries are given by the eigenvalues of $A^*A$, which correspond to the singular values of $A$ squared, and $Q\in\mathbb{C}^{s\times s}$ is unitary. Then 
\begin{align*}
    \|AB\|_{HS}^2&=\Tr(B^*A^*AB)\\
    &=\Tr(B^*Q^*\Lambda QB)\\
    &=\Tr((QB)^*\Lambda QB)\\
    &\geq \lambda_{\min}[A]^2\Tr((QB)^*QB)\\
    &=\lambda_{\min}[A]^2\Tr(B^*B)=\lambda_{\min}[A]^2\|B\|_{HS}^2,
\end{align*}    
which implies the claim.
\end{proof}

Notice that when $n=m=1$ the conditions for global hypoellipticity in Theorem \ref{theo-vector} coincide precisely with the conditions for global hypoellipticity of left-invariant continuous linear operators acting on scalar-valued functions, obtained in \cite[Theorem 20]{tese-Nicholas}.

Now consider the system of left-invariant continuous linear operators  $(P_{ji}):C^\infty(G,\mathbb{C}^n)\to C^\infty(G,\mathbb{C}^m)$, for $1\leq i\leq n$, $1\leq j\leq m$, and the associated system of equations:
    \begin{equation}\label{system1}
            \begin{cases}
                \begin{matrix}
        P_{11}u_1&+&\dots &+& P_{1n}u_n&=&f_1\\
        &&\vdots&&&&\vdots\\
        P_{m1}u_1&+&\dots&+&P_{mn}u_n&=&f_m,
    \end{matrix}
            \end{cases}
    \end{equation}
where $f_1,\dots, f_m\in C^\infty(G)$. The system of operators $(P_{ji})$ is said to be globally hypoelliptic if whenever system \eqref{system1} admits solution $u_1,\dots, u_n\in\mathcal{D}'(G)$ for $f_1,\dots, f_m\in C^\infty(G)$, this implies that $u_1,\dots, u_n\in C^\infty(G)$.

\noindent Note that system \eqref{system1} is equivalent to the vector-valued equation
\begin{equation*}
    Pu=f,
\end{equation*}
where $u(x)= (u_1(x),\dots,u_n(x))$,  $f(x)= (f_1(x),\dots,f_m(x))$, and $P: C^\infty(G,\mathbb{C}^n)\to C^\infty(G,\mathbb{C}^m)$ is given by
\begin{equation}\label{eq_operator_P}
    (Pu)_j=P_{j1}u_1+\dots+P_{jn}u_n,\quad 1\leq j\leq m.
\end{equation} 
We identify $P$ with the matrix of operators
\begin{equation*}
    P\equiv\begin{pmatrix}
        P_{11}&\dots&P_{1n}\\
        \vdots&\ddots&\vdots\\
        P_{m1}&\dots&P_{mn}
    \end{pmatrix}.
\end{equation*}
Notice that with these identifications, the concepts of global hypoellipticity for systems and for vector-valued functions coincide.

\noindent Furthermore, identity \eqref{eq_operator_P} implies that the symbol of $P$ as defined in \eqref{symbol} is given by 
\begin{equation*}
    \sigma_P(i,j,\xi)=\sigma_{P_{ji}}(\xi),
\end{equation*}
for every $1\leq i\leq n$, $1\leq j\leq m$ and $[\xi]\in\widehat{G}$, so the matrix-valued matrix $\sigma_P(\xi)$ is given by
\begin{equation*}
    \sigma_P(\xi)=\begin{pmatrix}
        \sigma_{P_{11}}(\xi)&\dots &\sigma_{P_{1n}}(\xi)\\
        \vdots&\ddots&\vdots\\
         \sigma_{P_{m1}}(\xi)&\dots &\sigma_{P_{mn}}(\xi)
    \end{pmatrix},
\end{equation*}
that is, 
\begin{equation*}
    \begin{pmatrix}
        \sigma_{P_{11}}(\xi)&\dots &\sigma_{P_{1n}}(\xi)\\
        \vdots&\ddots&\vdots\\
         \sigma_{P_{m1}}(\xi)&\dots &\sigma_{P_{mn}}(\xi)
    \end{pmatrix}\begin{pmatrix}
        \widehat{u_1}(\xi)\\\vdots\\\widehat{u_n}(\xi)
    \end{pmatrix}=\begin{pmatrix}
        \widehat{f_1}(\xi)\\\vdots\\\widehat{f_m}(\xi)
    \end{pmatrix},
\end{equation*}
for every $[\xi]\in\widehat{G}$. In particular, the smallest singular value of $\sigma_P(\xi)$, can then be written as 
\begin{equation}\label{m_xi-for-systems}
     \lambda_{\min}[\sigma_P(\xi)]= \min\left\{\left(\sum_{j=1}^{n}\left\|\sum_{i=1}^{m}
         \sigma_{P_{ji}}(\xi)v(i,\xi)\right\|_{2}^2\right)^{\frac{1}{2}}:\,v(i,\xi)\in \mathbb{C}^{d_\xi},\,\sum_{i=1}^{m}\|v(i,\xi)\|_2^2=1\right\}.
    \end{equation}
Therefore, from Theorem \ref{theo-vector} and the previous considerations, we obtain the following corollaries.
\begin{corollary}\label{theo-systems}
    Let  $n,m\in\N$ and $(P_{ji}):C^\infty(G,\mathbb{C}^n)\to C^\infty(G,\mathbb{C}^m)$ be a $m \times n$ system of left-invariant continuous linear operators.  Then the system $(P_{ji})$ is globally hypoelliptic if and only if there exist $k\in\R$ and $C>0$ such that
    \begin{equation*}
        \lambda_{\min}[\sigma_P(\xi)]\geq C\langle\xi\rangle^k,
    \end{equation*}
    for all but finitely many $[\xi]\in\widehat{G}$, where $\lambda_{\min}[\sigma_P(\xi)]$ denotes the smallest singular value of the block matrix 
    \begin{equation*}
    \sigma_P(\xi)=\begin{pmatrix}
        \sigma_{P_{11}}(\xi)&\dots &\sigma_{P_{1n}}(\xi)\\
        \vdots&\ddots&\vdots\\
         \sigma_{P_{m1}}(\xi)&\dots &\sigma_{P_{mn}}(\xi)
    \end{pmatrix}\in (\mathbb{C}^{d_\xi\times d_\xi})^{m\times n},
\end{equation*}
for every $[\xi]\in\widehat{G}$.
\end{corollary}
\begin{corollary}  
    Let $r,n,m\in\N$ and $(P_{ji}):C^\infty(\T^r,\mathbb{C}^n)\to C^\infty(\T^r,\mathbb{C}^m)$ be a $m\times n$ system of Fourier multipliers. Then the system $(P_{ji})$ is globally hypoelliptic if and only if there exist $k\in\R$ and $C>0$ such that
    \begin{equation*}
        \lambda_{\min}\left[\begin{pmatrix}
        {P_{11}}(\xi)&\dots &{P_{1n}}(\xi)\\
        \vdots&\ddots&\vdots\\
         {P_{m1}}(\xi)&\dots &{P_{mn}}(\xi)
    \end{pmatrix}\right]\geq C(1+\|\xi\|_2^2)^{k/2},
    \end{equation*}
    for all but finitely many $\xi\in\Z^r$, where $\xi\mapsto P_{ji}(\xi)\in\mathbb{C}$ denotes the symbol of the operator $P_{ji}$.
\end{corollary}
Since we can relate the smallest singular value of a block matrix and the singular values of its blocks, we can relate the singular values of the symbols $\sigma_{P_{ji}}$ and the global hypoellipticity of the system $(P_{ji})$. Below we present a few particular cases which illustrate this connection.

In \cite{errata} (in a correction of \cite{PiazzaPoliti}) the author proved that for an $\ell\times \ell$ complex matrix $A$, its smallest singular value is bounded from below by
\begin{equation*}
    \lambda_{\min}[A]\geq |\det A|{\Bigg(}\frac{\ell-1}{\|A\|_{\HS}^2}\Bigg)^{\frac{\ell-1}{2}}.
\end{equation*}
Consequently, for a system $(P_{ji})$ such as Corollary \ref{theo-systems} where $m=n\geq 2$, we have that
\begin{align*}
    \lambda_{\min}[\sigma_P(\xi)]&\geq |\det \sigma_P(\xi)|\left(\frac{m\cdot d_\xi-1}{\|\sigma_{P}(\xi)\|_{\HS}^2}\right)^{\frac{m\cdot d_\xi-1}{2}}\\&= |\det \sigma_P(\xi)|({m\cdot d_\xi-1})^{\frac{m\cdot d_\xi-1}{2}}\left({\sum_{i,j=1}^m\|\sigma_{P_{ji}}(\xi)\|_{\HS}^2}\right)^{-\frac{m\cdot d_\xi-1}{2}}\\
    &\geq e^{-1/(2e)}|\det \sigma_P(\xi)|\left({\sum_{i,j=1}^m\|\sigma_{P_{ji}}(\xi)\|_{\HS}^2}\right)^{-\frac{m\cdot d_\xi-1}{2}},
\end{align*}
for every $[\xi]\in\widehat{G}$. Hence if 
\begin{equation*}
    |\det \sigma_P(\xi)|\left({\sum_{i,j=1}^m\|\sigma_{P_{ji}}(\xi)\|_{\HS}^2}\right)^{-\frac{m\cdot d_\xi-1}{2}}\geq C\langle\xi\rangle^{k},
\end{equation*}
for all but finitely many $[\xi]\in \widehat{G}$, by Corollary \ref{theo-systems} we conclude that the system $(P_{ji})$ is globally hypoelliptic.  Notice that for any $N\times N$ square matrix $A$ we have $\|A\|_{HS}^2\leq N\|A\|_{\operatorname{op}}^2$, and since each $P_{ji}$ acts continuously on smooth functions, without loss of generality we may assume that there exists $\tau_P\in\R$ such that every symbol $\sigma_{P_{ji}}(\xi)$ is in the Ruzhansky-Turunen symbol class $\mathscr{S}^{\tau_P}_{0,0}(G\times \widehat{G})$, in other words there exist $C_{P}=\max\{C_{P_{ji}}\}>0,\tau_{P}=\max\{\tau_{P_{ji}}\}\in\R$ such that
\begin{equation}\label{ineq-symbol-class}
    \|\sigma_{P_{ji}}(\xi)\|_{\operatorname{op}}\leq C_{P}\langle\xi\rangle^{\tau_{P}},
\end{equation}
for every $[\xi]\in\widehat{G}$.
Therefore we have that
\begin{align*}
     \left({\sum_{i,j=1}^m\|\sigma_{P_{ji}}(\xi)\|_{\HS}^2}\right)^{-\frac{m\cdot d_\xi-1}{2}}&\geq \left(d_\xi{\sum_{i,j=1}^m\|\sigma_{P_{ji}}(\xi)\|_{\operatorname{op}}^2}\right)^{-\frac{m\cdot d_\xi-1}{2}}\\
     &\geq\left(C_G\langle\xi\rangle^{\frac{d}{2}}m^2 C_P^2\langle\xi\rangle^{2\tau_P}\right)^{-\frac{m\cdot d_\xi-1}{2}},
\end{align*}
 where we have used the fact that $d_\xi\leq C_G\langle\xi\rangle^{\frac{d}{2}}$, for some $C_G>0$, $d=\dim(G)$ and every $[\xi]\in\widehat{G}$, which follows from the Weyl formula for the counting function of the eigenvalues of the first-order elliptic operator $(\operatorname{Id}+\mathcal{L}_G)^{1/2}$. 

Next we consider two cases: $\tau_P<-\frac{d}{4}$ or $d_\xi\leq K$ for some $K\in\N$ and all $[\xi]\in\widehat{G}$.

If we assume that $\tau_P<-\frac{d}{4}$, then for large enough $\langle\xi\rangle$ (and so for all but finitely many $[\xi]\in\widehat{G}$) we have that $C_G\langle\xi\rangle^{\frac{d}{2}}m^2 C_P^2\langle\xi\rangle^{2\tau_P}<1$.
Hence in this case
\begin{equation*}
     \left({\sum_{i,j}\|\sigma_{P_{ji}}(\xi)\|_{\HS}^2}\right)^{-\frac{m\cdot d_\xi-1}{2}}\geq 1.
\end{equation*}

Alternatively, if $d_\xi\leq K$, for some  $K\in\N$ and all $[\xi]\in\widehat{G}$, then  
\begin{align*}
     \left({\sum_{i,j=1}^m\|\sigma_{P_{ji}}(\xi)\|_{\HS}^2}\right)^{-\frac{m\cdot d_\xi-1}{2}}&\geq \left(d_\xi{\sum_{i,j=1}^m\|\sigma_{P_{ji}}(\xi)\|_{\operatorname{op}}^2}\right)^{-\frac{m\cdot d_\xi-1}{2}}\\
     &\geq\left(Km^2 C_P^2\langle\xi\rangle^{2\tau_P}\right)^{-\frac{m\cdot d_\xi-1}{2}},
\end{align*}
for every $[\xi]\in\widehat{G}$. Note that if $Km^2C_P^2\langle\xi\rangle^{2\tau_P}\geq 1$ then
\begin{align*}
    \left(Km^2 C_P^2\langle\xi\rangle^{2\tau_P}\right)^{-\frac{m\cdot d_\xi-1}{2}}&\geq \left(Km^2 C_P^2\langle\xi\rangle^{2\tau_P}\right)^{-\frac{mK-1}{2}}\\
    &=\left(Km^2 C_P^2\right)^{-\frac{mK-1}{2}}\langle\xi\rangle^{-\tau_P(mK-1)},
\end{align*}
and if $0<Km^2C_P^2\langle\xi\rangle^{2\tau_P}< 1$, then 
\begin{align*}
    \left(Km^2 C_P^2\langle\xi\rangle^{2\tau_P}\right)^{-\frac{m\cdot d_\xi-1}{2}}&> 1,
\end{align*}
hence 
\begin{align*}
     \left({\sum_{i,j=1}^m\|\sigma_{P_{ji}}(\xi)\|_{\HS}^2}\right)^{-\frac{m\cdot d_\xi-1}{2}}&\geq \min\left\{1,\left(Km^2 C_P^2\right)^{-\frac{mK-1}{2}}\langle\xi\rangle^{-\tau_P(mK-1)}\right\}\\
     &\geq \begin{cases}
         \left(Km^2 C_P^2\right)^{-\frac{mK-1}{2}}\langle\xi\rangle^{-\tau_P(mK-1)}&\text{if }\tau_P>0,\\
         \min\{1,\left(Km^2 C_P^2\right)^{-\frac{mK-1}{2}}\}&\text{if }\tau_P=0,\\
         1&\text{if }\tau_P<0,
     \end{cases}
\end{align*}
for sufficiently large $\langle\xi\rangle$, and so for all but finitely many $[\xi]\in\widehat{G}$.

Therefore by Corollary \ref{theo-systems} and our previous considerations, in both cases we conclude that if there exist $C>0$ and $k\in\R$ such that 
\begin{equation*}
    |\det(\sigma_P(\xi)|\geq C\langle \xi \rangle ^k
\end{equation*}
for all but finitely many $[\xi]\in\widehat{G}$, the system $(P_{ji})$ is globally hypoelliptic.

In conclusion, we have proved the following.
\begin{theorem}
    Let $m\in\N$ and $P=(P_{ji}):C^\infty(G,\mathbb{C}^m)\to C^\infty(G,\mathbb{C}^m)$ be a square system of left-invariant continuous linear operators. Then the system $(P_{ji})$ is globally hypoelliptic if there exist $C>0$ and $k\in\R$ such that
    \begin{equation*}
        |\det \sigma_P(\xi)|\geq C\langle\xi\rangle^k,
    \end{equation*}
    for all but finitely many $[\xi]\in\widehat{G}$
    and either every 
    $P_{ji}$ is of order less than $-\frac{\dim(G)}{4}$, for $1\leq i,j\leq m$ or there exists $K\in \N$ such that $d_\xi\leq K$, for every $[\xi]\in\widehat{G}$.
\end{theorem}
\begin{corollary}
    Let $P:C^\infty(G)\to C^\infty(G)$ be a left-invariant continuous linear operator. Then the system $P$ is globally hypoelliptic if there exist $C>0$ and $k\in\R$ such that
    \begin{equation*}
        |\det \sigma_P(\xi)|\geq C\langle\xi\rangle^k,
    \end{equation*}
    for all but finitely many $[\xi]\in\widehat{G}$
    and either $P$ is of order less than $-\frac{\dim(G)}{4}$, or there exists $K\in \N$ such that $d_\xi\leq K$, for every $[\xi]\in\widehat{G}$.
\end{corollary}
\begin{corollary}
    Let $m,r\in\N$ and $P=(P_{ji}):C^\infty(\T^r,\mathbb{C}^m)\to C^\infty(\T^r,\mathbb{C}^m)$ be a system of Fourier multipliers. Then the system $(P_{ji})$ is globally hypoelliptic if there exist $C>0$ and $k\in\R$ such that
    \begin{equation*}
        \left|\det \begin{pmatrix}
        {P_{11}}(\xi)&\dots &{P_{1n}}(\xi)\\
        \vdots&\ddots&\vdots\\
         {P_{m1}}(\xi)&\dots &{P_{mn}}(\xi)
    \end{pmatrix}\right|\geq C(1+\|\xi\|_2^2)^{k/2},
    \end{equation*}
    for all but finitely many $\xi\in\Z^r$, where $\xi\mapsto P_{ji}(\xi)$ denotes the symbol of $P_{ji}$.
\end{corollary}

Next, note that if the system $(P_{ji})$ is diagonal, then its global hypoellipticity is determined trivially by the global hypoellipticity of each $P_{jj}$, $1\leq j\leq m$. More precisely, $(P_{ji})$ is globally hypoelliptic if and only if $P_{jj}$ is globally hypoelliptic, for every $1\leq j\leq m$. With this in mind, next consider the following type of systems.

Let $m\geq 2$ and consider $(P_{ji})$ a $m\times m$ system of continuous linear operators such as in Corollary \ref{theo-systems}, and such that $\sigma_P(\xi)$ is block diagonally dominant by rows and columns, for all but finitely many $[\xi]\in\widehat{G}$. A square block matrix $A=(A_{ji})$ is said to be block diagonally dominant by rows and columns if its diagonal blocks are non-singular and it satisfies
\begin{equation}\label{block_dominant}
    \|A_{\ell\ell}^{-1}\|_{\max}^{-1}>\sum_{i\neq \ell}\|A_{\ell i}\|_{\max}\quad\text{ and }\quad \|A_{\ell\ell}^{-1}\|_{\max}^{-1}>\sum_{j\neq \ell}\|A_{j\ell}\|_{\max},
\end{equation}
for all $\ell$. Notice that using the inequality $\|B\|_{\max}\leq \|B\|_{\operatorname{op}}$, which holds for any complex matrix $B$, we have that a sufficient condition for \eqref{block_dominant} to hold is that
\begin{equation}\label{block_dominant2}
    \|A_{\ell\ell}^{-1}\|_{\operatorname{op}}^{-1}>\sum_{i\neq \ell}\|A_{\ell i}\|_{\operatorname{op}}\quad\text{ and }\quad \|A_{\ell\ell}^{-1}\|_{\operatorname{op}}^{-1}>\sum_{j\neq \ell}\|A_{j\ell}\|_{\operatorname{op}}.
\end{equation}
In \cite{VarahLowerBound} the author proved that for a $m\times m$ square block matrix $A$ that is block diagonally dominant by rows and columns, we have that  
\begin{equation*}
    \lambda_{\min}[A]\geq \sqrt{\alpha\beta},
\end{equation*}
where 
\begin{equation*}
    \alpha=\min_{1\leq \ell\leq m}\left\{\|A_{\ell\ell}^{-1}\|_{\max}^{-1}-\sum_{i\neq \ell}\|A_{\ell i}\|_{\max}\right\}\quad \text{and}\quad \beta=\min_{1\leq \ell\leq m}\left\{\|A_{\ell\ell}^{-1}\|_{\max}^{-1}-\sum_{j\neq \ell}\|A_{j \ell}\|_{\max}\right\}.
\end{equation*}
Again using inequality $\|B\|_{\max}\leq \|B\|_{\operatorname{op}}$, and the fact that $\|B^{-1}\|_{\operatorname{op}}^{-1}=\lambda_{\min}[B]$, which holds for any non-singular matrix $B$, we obtain that
\begin{align*}
     \alpha\geq\alpha^*= \min_{1\leq \ell\leq m}\left\{\lambda_{\min}[A_{\ell\ell}]-\sum_{i\neq \ell}\|A_{\ell i}\|_{\operatorname{op}}\right\},
     \end{align*}
     \begin{align*}
         \beta\geq \beta^* = \min_{1\leq \ell\leq m}\left\{\lambda_{\min}[A_{\ell\ell}]-\sum_{j\neq \ell}\|A_{j\ell}\|_{\operatorname{op}}\right\},
\end{align*}
where $n\in\N$ is the common dimension of the blocks of $A$. Hence
\begin{equation*}
    \lambda_{\min}[A]\geq \sqrt{\alpha^*\beta^*},
\end{equation*}
as well. Then using the estimates above we have that
\begin{align*}
    \lambda_{\min}[\sigma_P(\xi)]\geq \sqrt{\alpha^*_\xi\beta^*_\xi},
\end{align*}
for all but finitely many $[\xi]\in\widehat{G}$, where 
\begin{align*}
    \alpha^*_\xi&=\min_{1\leq \ell\leq m}\left\{\lambda_{\min}[\sigma_{P_{\ell\ell}}(\xi)]-\sum_{i\neq \ell}\|\sigma_{P_{\ell i}}(\xi)\|_{\operatorname{op}}\right\},\\
        \beta^*_\xi&=\min_{1\leq \ell\leq m}\left\{\lambda_{\min}[\sigma_{P_{\ell\ell}}(\xi)]-\sum_{j\neq \ell}\|\sigma_{P_{j \ell}}(\xi)\|_{\operatorname{op}}\right\}.
\end{align*}
As before, without loss of generality we may assume that every $P_{i\ell}$, $P_{\ell i}$, $i\neq \ell$ belongs to the same symbol class $\mathscr{S}_{0,0}^{\tau_\ell}(G\times \widehat{G})$, for $1\leq \ell\leq m$.
Then note that
\begin{align}
       \lambda_{\min}[\sigma_{P_{\ell\ell}}(\xi)]-\sum_{i\neq \ell}\|\sigma_{P_{\ell i}}(\xi)\|_{\operatorname{op}}\notag&\geq \lambda_{\min}[\sigma_{P_{\ell\ell}}(\xi)]-\sum_{i\neq \ell}\max_{i\neq \ell}\{C_{P_{\ell i}}\}\langle\xi\rangle^{\tau_{\ell}}\notag\\
       &\geq \left( \lambda_{\min}[\sigma_{P_{\ell\ell}}(\xi)]-(m-1)\max_{i\neq \ell}\{C_{P_{\ell i}}\}\langle\xi\rangle^{\tau_\ell}\right),\label{ineq-dominant}
\end{align}
for $1\leq \ell\leq m$. Let $C'_\ell\doteq \max_{i\neq \ell}\{C_{P_{\ell i}}\}>0$. If there exist $C_\ell>0$ and $k_\ell\in\R$ such that
\begin{equation*}
    \lambda_{\min}[\sigma_{P_{\ell\ell}}(\xi)]\geq C_\ell\langle\xi\rangle^{k_\ell},
\end{equation*}
for $1\leq \ell\leq m$ and all but finitely many $[\xi]\in\widehat{G}$ (so that in particular every $P_{\ell\ell}$ is globally hypoelliptic, by \cite[Theorem 3.1]{CardKow}), then
\begin{align*}
     \lambda_{\min}[\sigma_{P_{\ell\ell}}(\xi)]-(m-1)C_\ell'\langle\xi\rangle^{\tau_\ell}\geq  C_\ell\langle\xi\rangle^{k_\ell}\left(1- C_\ell^{-1}C_\ell' (m-1)\langle\xi\rangle^{\tau_\ell-k_\ell}\right),
\end{align*}
for all but finitely many $[\xi]\in\widehat{G}$. So, if $k_\ell>\tau_\ell$, for large enough $\langle \xi\rangle$ (and so for all but finitely many $[\xi]\in \widehat{G}$), we have that
\begin{equation*}
    1- C_\ell^{-1} C'_\ell(m-1)\langle\xi\rangle^{\tau_\ell-k_\ell}\geq \frac{1}{2},
\end{equation*}
and so applying this inequality to \eqref{ineq-dominant} yields
\begin{align*}
    \alpha^*_\xi\geq \min_{1\leq \ell\leq m}\left\{\frac{1}{2}C_\ell\langle\xi\rangle^{k_\ell}\right\}&\geq \frac{1}{2} \min_{1\leq \ell\leq m}\{C_\ell\}\langle\xi\rangle^{\min_{1\leq \ell\leq m}\{k_\ell\}}\\
    &\geq \frac{C}{2}\langle\xi\rangle^{k},
\end{align*}
where $C= \min_{1\leq \ell\leq m}\{C_\ell\}$ and $k=\min_{1\leq \ell\leq m}\{k_\ell\}$, for all but finitely many $[\xi]\in\widehat{G}$.
Evidently, analogous arguments result in the same estimate for for $\beta^*_\xi$, therefore under these conditions we have that
\begin{equation*}
    \lambda_{\min}[\sigma_P(\xi)]\geq  \frac{C}{2}\langle\xi\rangle^{k},
\end{equation*}
for all but finitely many $[\xi]\in\widehat{G}$. In summary, we have proved the following.
\begin{theorem}\label{thoerem_diagonal}
     Let $m\in\N$ and $P=(P_{ji}):C^\infty(G,\mathbb{C}^m)\to C^\infty(G,\mathbb{C}^m)$ be a system of left-invariant continuous linear operators such that $\sigma_P(\xi)$ is block diagonally dominant by rows and columns for all but finitely many $[\xi]\in\widehat{G}$. Also let $\tau_\ell\in\R$ be large enough so that $P_{j\ell},P_{\ell i}$ have order at most $\tau_\ell$ for every $1\leq i,j,\ell\leq m$, $i,j\neq \ell$.
     Then the system $(P_{ji})$ is globally hypoelliptic if there exist $C_\ell>0$ and $k_\ell\in\R$, $k_\ell>\tau_\ell$ such that
    \begin{equation*}
        \lambda_{\min}[\sigma_{P_{\ell\ell}}(\xi)]\geq C_\ell\langle\xi\rangle^{k_\ell},
    \end{equation*}
    for $1\leq \ell\leq m$ and for all but finitely many $[\xi]\in\widehat{G}$.
\end{theorem}


\begin{corollary}\label{coro_toro_diag}
    Let $m,r\in\N$ and $P=(P_{ji}):C^\infty(\T^r,\mathbb{C}^m)\to C^\infty(\T^r,\mathbb{C}^m)$ be a system of Fourier multipliers such that $\sigma_P(\xi)$ is diagonally dominant by rows and columns for all but finitely many $\xi\in\Z^r$, in the sense that
  \begin{equation}\label{diag_dom}
    |P_{\ell\ell}(\xi)|>\sum_{i\neq \ell}|P_{\ell i}(\xi)|\quad\text{ and }\quad |P_{\ell\ell}(\xi)|>\sum_{j\neq \ell}|P_{j\ell}(\xi)|,
\end{equation}
for all but finitely many $\xi\in\Z^r$.
    Also let $\tau_\ell\in\R$ be large enough so that every $P_{\ell i},P_{j \ell}$, $i,j\neq \ell$ has order at most $\tau_\ell$, for $1\leq i,j,\ell\leq m$. 
    Then the system $(P_{ji})$ is globally hypoelliptic if there exist $C_\ell>0$ and $k_\ell\in\R$, $k_\ell>\tau_\ell$ such that
    \begin{equation}\label{cond_coro}
        |{P_{\ell\ell}}(\xi)|\geq C_\ell(1+\|\xi\|_2^2)^{k_\ell/2},
    \end{equation}
    for $1\leq \ell\leq m$ and all but finitely many $\xi\in\Z^r$, where $\xi\mapsto P_{ji}(\xi)\in\mathbb{C}$ denotes the symbol of $P_{ji}$, $1\leq i,j\leq m.$
\end{corollary}

\begin{obs}
    Recall that in the case of a $m\times m$ diagonal  system of continuous linear operators $(P_{ji})$, a sufficient (and necessary) condition for global hypoellipticity is that $P_{\ell\ell}$ is globally hypoelliptic, for $1\leq \ell\leq m$. By Theorem 3.1 in \cite{CardKow} this implies that there must exist $C>0$ and $k\in\R$ such that
    \begin{equation*}
         \lambda_{\min}[\sigma_{P_{\ell\ell}}(\xi)]\geq C\langle\xi\rangle^k,
    \end{equation*}
    for $1\leq \ell\leq m$ and all but finitely many $[\xi]\in\widehat{G}$. The results above can then be seen as a generalization of the diagonal case, where a more restrictive condition appears requiring some ``stronger" regularity of the operators on the diagonal depending on the order of the operators outside the diagonal.
\end{obs}

\begin{example}
     Consider a $m\times m$ system of left-invariant continuous linear operators $(P_{ji})$ on a compact Lie group $G$ such that
     $P_{\ell\ell}$ is elliptic of order $k_\ell\in\R$ and $P_{\ell i},P_{i\ell}$ have order less than $k_\ell$ for every $i,j\neq \ell$, $1\leq \ell\leq m$. Then  the system $(P_{ji})$ is globally hypoelliptic. Indeed, since $P_{\ell\ell}$ is elliptic of order $k_\ell\in\R$, by definition we have that 
     \begin{equation*}
         \|\sigma_{P_{\ell\ell}}(\xi)^{-1}\|_{\operatorname{op}}\leq C\langle\xi\rangle^{-k_\ell}, 
     \end{equation*}
     for all but finitely many $[\xi]\in\widehat{G}$. Consequently for all such $[\xi]$ this implies that 
     \begin{equation*}
         \|\sigma_P(\xi)\|_{\operatorname{op}}\geq \frac{1}{C}\langle \xi\rangle^{k_\ell},
     \end{equation*}
     so by the sufficient condition \eqref{block_dominant2} the matrices $\sigma_P(\xi)$ are block diagonally dominant, and 
     \begin{equation*}
         \lambda_{\min}[\sigma_P(\xi)]= \|\sigma_{P_{\ell\ell}}(\xi)^{-1}\|_{\operatorname{op}}^{-1}\geq  \frac{1}{C}\langle \xi\rangle^{k_\ell},
     \end{equation*}
     for all but finitely many $[\xi]\in\widehat{G}$. Hence the system satisfies the conditions of Theorem \ref{thoerem_diagonal}, and therefore is globally hypoelliptic  as claimed.
\end{example}

\begin{example}
     Let $m,r\in\N$ and $P=(P_{ji}):C^\infty(\T^r,\mathbb{C}^m)\to C^\infty(\T^r,\mathbb{C}^m)$ be a system of Fourier multipliers such that $\sigma_P(\xi)$ is diagonally dominant by rows and columns for all but finitely many $\xi\in\Z^r$. Also let $\tau_\ell\in\R$ be large enough so that every $P_{\ell i},P_{j \ell}$, $i,j\neq \ell$ has order at most $\tau_\ell$, for $1\leq i,j,\ell\leq m$. Following the definitions in \cite{kkfinite}, let $\tau_\ell'$ denote the essential order of $P_{\ell\ell}$ (note that $\tau_\ell'\geq \tau_\ell$ by \eqref{diag_dom}). If $\tau_\ell'>\tau_\ell$, and $P_{\ell\ell}$ is globally hypoelliptic with $r_\ell<\tau_\ell'-\tau_\ell$ loss of derivatives (as defined in  \cite{kkfinite}), for every $1\leq \ell\leq m$, then by Corollary \eqref{cond_coro} the system $(P_{ji})$ is globally hypoelliptic. 
\end{example}



 Another class of system of operators which has been intensively studied in the last few years (see for instance \cite{BergamascoGH,systems-gelfand,overdetermined,gevrey-torus,stronglyinvariantsyst}) corresponds to the case where $n=1$. More precisely, consider a system of left-invariant continuous linear operators $P_{j1}:C^\infty(G)\to C^\infty(G)$, for $1\leq j\leq m$, and the associated system of equations:
\begin{equation}\label{system2}
    \begin{cases}
        \begin{matrix}
    P_{11}u_1=f_1\\
    \vdots\\
    P_{m1}u_1=f_m,
\end{matrix}
    \end{cases}
\end{equation}
where $f_1,\dots, f_m\in C^\infty(G)$. The system $(P_{j1})$ is said to be globally hypoelliptic if whenever the system \eqref{system2} admits solution $u_1\in\mathcal{D}'(G)$ for $f_1,\dots, f_m\in C^\infty(G)$, this implies that $u_1\in C^\infty(G)$.

To simplify the notation, in this case we denote $P_j= P_{j1}$, $1\leq j\leq m$ and $u= u_1$. 
Then system \eqref{system2} is equivalent to  the vector-valued equation
\begin{equation*}
    Pu=f,
\end{equation*}
where $P: C^\infty(G,\mathbb{C}^1)\to C^\infty(G,\mathbb{C}^m)$ is given by
\begin{equation}\label{eq_operator_P2}
    (Pu)_j=P_{j}u,\quad 1\leq j\leq m,
\end{equation} 
and we identify $P$ with the $m\times 1$ matrix of operators
\begin{equation*}
    P=\begin{pmatrix}
        P_1\\
        \vdots\\
        P_{m}
    \end{pmatrix}.
\end{equation*}
Then $\lambda_{\min}[\sigma_P(\xi)]$  can be written as 
\begin{align*}
      \lambda_{\min}[\sigma_P(\xi)]&= \min\left\{\left(\sum_{j=1}^{m}\left\|
         \sigma_{P_{j}}(\xi)v(\xi)\right\|_{2}^2\right)^{\frac{1}{2}}\bigg|\,v(\xi)\in \mathbb{C}^{d_\xi},\,\|v(\xi)\|_2=1\right\}\\
         &=\min\left\{\left\|
         [\sigma_{P_{j}}(\xi)v(\xi)]_{j=1}^m\right\|_{2} : \,v(\xi)\in \mathbb{C}^{d_\xi},\,\|v(\xi)\|_2=1\right\},
    \end{align*}
    for every $[\xi]\in\widehat{G}$, where in the last line the first $\|\cdot\|_2$ denotes the Euclidean norm in $(\mathbb{C}^{d_\xi})^m=\C^{md_\xi}$.

    Note that in the case of a block column matrix $A=(A_\ell)_{\ell=1}^m$, we have that $\lambda_{\min}[A_\ell]\leq \lambda_{\min}[A]$, for $1\leq \ell\leq m$. Hence
\begin{equation*}
    \lambda_{\min}[\sigma_P(\xi)]\geq \max_{1\leq j\leq m}\lambda_{\min}[\sigma_{P_j}(\xi))],
\end{equation*}
for every $[\xi]\in\widehat{G}$.

Therefore we conclude that if there exist $k\in\R$ and $C>0$ such that
\begin{equation*}
    \max_{1\leq j\leq m}\lambda_{\min}[\sigma_{P_{j}}(\xi)]\geq C\langle\xi\rangle^k,
\end{equation*}
for all but finitely many $[\xi]\in\widehat{G}$, then the system $(P_j)$ is globally hypoelliptic.

Finally, note that in the case where $G=\T^r$, then 
\begin{equation*}
    \lambda_{\min}[\sigma_P(\xi)]=\sqrt{|P_{1}(\xi)|^2+\dots+|P_{m}(\xi)|^2},
\end{equation*}
so 
\begin{equation*}
    \max_{1\leq \ell\leq m}|P_{\ell}(\xi)|\leq \lambda_{\min}[\sigma_P(\xi)]\leq \sqrt{m}\max_{1\leq \ell\leq m}|P_{\ell}(\xi)|,
\end{equation*}
hence the system $(P_j)$ is globally hypoelliptic if and only if there exist $k\in\R$ and $C>0$ such that
\begin{equation*}
    \max_{1\leq j\leq m}\lambda_{\min}[\sigma_{P_{j}}(\xi)]\geq C\langle\xi\rangle^k,
\end{equation*}
for all but finitely many $[\xi]\in\widehat{G}$.

 This last result can also be seen as a consequence of the results in \cite{diag-systems}, where the authors obtained necessary and sufficient conditions for global hypoellipticity for such systems on general compact Lie groups, by requiring some extra assumptions on the symbol of every $P_j$.


\bibliographystyle{plain}
\bibliography{references}

\end{document}